\documentclass[11pt, a4paper]{amsart}

\usepackage{amsmath, amsthm, amssymb,fullpage}

\usepackage{color}
\usepackage{xcolor}

\usepackage{enumitem}
\usepackage{hyperref}

\theoremstyle{plain}
\newtheorem{theorem}{Theorem}[section]
\newtheorem{corollary}[theorem]{Corollary}
\newtheorem{lemma}[theorem]{Lemma}
\newtheorem{proposition}[theorem]{Proposition}

\theoremstyle{definition}
\newtheorem{definition}[theorem]{Definition}

\theoremstyle{remark}
\newtheorem{remark}[theorem]{Remark}

\newtheorem{example}[theorem]{Example}

\numberwithin{equation}{section}

\renewcommand{\P}{\mathbb P}
\newcommand{\C}{\mathbb C}

\newcommand{\K}{\mathbb K}
\newcommand{\I}{\mathbb I}
\newcommand{\T}{\mathbb T}

\newcommand{\PGL}{\mathrm {PGL}}

\newcommand{\Cc}{\mathcal C}

\usepackage{bbold}

\newcommand{\B}{\mathbb B}

\usepackage{tikz-cd}

\newcommand{\id}{{\rm id}}

\newcommand{\FS}{{\text{\rm \tiny FS}}}

\newcommand{\ddc}{{dd^c}}

\newcommand{\dbar}{{\overline\partial}}

\renewcommand{\H}{\mathbb{H}}

\usepackage{stmaryrd}

\usepackage{enumitem}
\usepackage{hyperref}

\subjclass[2010]{37F80
(primary), 32U05, 32H50, 37A25, 60F05 (secondary)}

\keywords{Complex H\'enon maps, 
 Exponential mixing of all orders, 
 Central Limit Theorem}

\begin{document} 

\hyphenpenalty=10000

\title{Every complex H\'enon map is exponentially mixing
\\ of all orders  and satisfies the CLT}

\begin{author}[F.~Bianchi]{Fabrizio Bianchi}
\address{ 
CNRS, Univ. Lille, UMR 8524 - Laboratoire Paul Painlev\'e, F-59000 Lille, France}
  \email{fabrizio.bianchi$@$univ-lille.fr}
\end{author}

\begin{author}[T.C.~Dinh]{Tien-Cuong Dinh}
\address{National University of Singapore, Lower Kent Ridge Road 10,
Singapore 119076, Singapore}
\email{matdtc$@$nus.edu.sg }
\end{author}

\maketitle

\begin{abstract}
 We show that the  measure of maximal entropy
 of every
complex H\'enon map is 
 exponentially mixing of all orders for H\"older observables. 
As a consequence,
 the Central Limit  Theorem 
 holds 
 for
all 
 H\"older observables.
\end{abstract}

\bigskip

\noindent
{\bf Notation.} 
The pairing $\langle\cdot, \cdot\rangle$
 is used for the integral of a function with respect to a measure or more
generally the value of a current at a test form.
 By
 $(p,p)$-currents
 we mean 
 currents of bi-digree $(p,p)$.
 Given $k\geq 1$,
we denote by $\omega_{\FS}$ the Fubini-Study form on $\mathbb P^k = \mathbb P^k(\mathbb C)$.
The mass of a positive closed $(p,p)$-current $R$ on $\mathbb P^k$
is equal to $\langle R,\omega_{\FS}^{k-p}\rangle$ and is denoted by $\|R\|$.
The notations $\lesssim$ and $\gtrsim$ 
stand for inequalities up to
a multiplicative constant.
If $R$ and $S$
are two real currents of the same bi-degree, we write $|R|\leq S$ when
$S\pm R\geq 0$.
 Observe that this forces $S$ to be positive.

\section{Introduction}

 H{\'e}non maps are among
 the most studied 
  dynamical systems that exhibit interesting chaotic behaviour.
They were introduced by Michel Hénon 
in the real setting
as a simplified model of the Poincaré section for
the Lorenz model, see,
 e.g., \cite{BC91,Henon76}.
H{\'e}non maps are also actively studied in the complex setting,
where complex analysis offers additional powerful tools. 
The reader can find in
 the work of Bedford, Fornaess, Lyubich, Sibony, Smillie, and the second author
fundamental dynamical properties of these systems, see
\cite{BLS93,BS91,BS92,DS14rigidity,Fornaess96,FS92,Sibony99}
and the references therein.
 It is shown in \cite{BLS93} that the measure of maximal entropy of such systems is Bernoulli. In particular, it is 
mixing of all orders. On the other hand, the control of the speed of mixing
for general dynamical systems 
is a challenging problem, and usually one can obtain it only under strong hyperbolicity assumptions on the system.
The main goal of 
this work is 
prove that
 the measure of maximal entropy of any complex H\'enon map
is \emph{exponentially mixing of all orders} with respect to H{\"o}lder observables.
As a consequence, 
we also
 solve a long-standing question 
proving
 the Central Limit Theorem
 for all H{\"o}lder observables
 with respect to
  the maximal entropy measures of 
  complex
H{\'e}non maps.

\medskip

Let us first recall the following 
 general definition.

\begin{definition}\label{d:exp-mixing-all-orders}
Let $(X,f)$
be a dynamical system
and $\nu$ an $f$-invariant measure.
Let $(E, \|\cdot\|_E)$ be a normed space of real functions on $X$
with 
$\|\cdot \|_{L^p(\nu)} \lesssim \|\cdot \|_E$ for all $1\leq p < \infty$.
We say that $\nu$ is \emph{exponentially mixing of order $\kappa\in\mathbb N^*$} 
for 
observables in $E$ 
 if 
 there exist
constants
$C_{\kappa}>0$
and
 $0< \theta_{\kappa} < 1$
  such that,  for all
 $g_0, \dots, g_{\kappa}$ in $E$
 and integers $0=: n_0 \leq n_1 \leq \dots \leq n_{\kappa}$, 
 we have
\[
\Big|
\langle \nu,
 g_0 ( g_1 \circ  f^{n_1} )
\dots
 (g_{\kappa} \circ f^{n_{\kappa}})
\rangle
- \prod_{j=0}^{\kappa} \langle  \nu, g_j
\rangle
\Big|
\leq C_{\kappa}
\cdot \Big(
\prod_{j=0}^{\kappa} 
\|g_j\|_{E}
\Big) \cdot
\theta_{\kappa}^{\min_{0\leq j \leq \kappa-1} ( n_{j+1}- n_j)}.
\] 
We say that $\nu$ is \emph{exponentially mixing of all orders} 
for 
observables in $E$ 
 if 
it is exponentially mixing of order $\kappa$ for every $\kappa \in \mathbb N$.
\end{definition}

A recent major result by Dolgopyat, Kanigowski, and Rodriguez-Hertz
\cite{DKRH}
 ensures that, under suitable assumptions on the system,
the exponential mixing of order 1 implies that the system is Bernoulli. In particular, it implies the mixing of all orders.
(with no control on the rate of decay of correlation).
 It is a main open question whether 
the exponential mixing of order 1 implies the exponential mixing of all orders, see for instance \cite[Question 1.5]{DKRH}.

\medskip

Let now $f$ be a complex H\'enon map on $\mathbb C^2$. It is 
a polynomial diffeomorphism of $\mathbb C^2$. 
We can associate to $f$ its unique measure of maximal entropy $\mu$
\cite{BLS93,BS91,BS92,Sibony99}, see 
Section \ref{s:regular} for details. 
It was established by the second author in \cite{Dinh05} that such measure
if exponential mixing of order 1 for H{\"o}lder observables, 
see also  Vigny \cite{Vigny15} and Wu \cite{Wu22}. Similar
results were obtained
by Liverani \cite{Liverani95} 
in the case of uniformly hyperbolic diffeomorphisms
and Dolgopyat \cite{Dolgopyat98} for Anosov flows.

\begin{theorem}\label{t:henon-mixing}
Let $f$ be a complex H\'enon map and $\mu$
 its measure of maximal entropy. Then, for every $\kappa\in \mathbb N^*$,
  $\mu$
is exponential mixing 
of order $\kappa$ 
as in Definition \ref{d:exp-mixing-all-orders}
for
$\Cc^\gamma$
 observables ($0<\gamma\leq 2$),
with $\theta_\kappa = d^{-(\gamma/2)^{\kappa+1} /2}$.
\end{theorem}

For endomorphisms of $\mathbb P^k (\mathbb C)$, the 
exponential mixing for all orders 
for the measure
of maximal entropy and  H{\"o}lder
observables was established in \cite{DNS10}.
We recently proved such property 
for
a large class of invariant measures
with
 strictly positive Lyapunov exponents \cite{BD22}.
This was
done by constructing 
a suitable (semi-)norm on functions that 
turns the 
 so-called Ruelle-Perron-Frobenius operator (suitably normalized)
into a contraction. 
As far as we know, the present paper gives the first instance where
the exponential mixing of all orders is established for general 
dynamical systems with both positive and negative Lyapunov exponents.

\medskip

 The exponential mixing of all orders is one of the strongest properties in dynamics. It was recently
shown to imply a number of statistical properties, see for instance \cite{BG,DFL}. As an example,
a 
consequence of 
Theorem \ref{t:henon-mixing}
 is
 the following result.
Take $u\in L^1 (\mu)$.
As $\mu$ is ergodic, Birkhoff's 
ergodic theorem states that
\[
n^{-1} S_n (u):=
n^{-1} \big( u (x)+ u \circ f(x) + \dots + u \circ f^{n-1}(x) \big)\to  \langle\mu, u\rangle
\quad \mbox{ for } \mu-\mbox{a.e. } x \in X.
\]
We say that $u$ \emph{satisfies the Central Limit Theorem}
(CLT)
with variance $\sigma^2 \geq 0$ with respect to $\mu$
if $n^{-1/2} (S_n (u)  - n\langle\mu,u\rangle )\to \mathcal N(0,\sigma^2)$ in law,
 where $\mathcal N (0,\sigma^2)$
denotes the
(possibly degenerate, for $\sigma =0$)
 Gaussian distribution 
with mean 0 and variance $\sigma^2$, i.e., 
 for any interval $I \subset \mathbb R$ we have
\[
\lim_{n\to \infty} \nu \Big\{
\frac{S_n (u) - n\langle\mu,u\rangle }{\sqrt{n} } \in I
\Big\}
=
\begin{cases}
1 \mbox{ when } I \mbox{ is of the form } I=(-\delta,\delta)  & \mbox{ if } \sigma^2=0,\\
\frac{1}{\sqrt{2\pi\sigma^2}}\displaystyle\int_I e^{-t^2 / (2\sigma^2)} dt &  \mbox{ if }
\sigma^2 >0.
\end{cases} \]

\medskip

By \cite{BG}, the following is then a 
 consequence of Theorem
 \ref{t:henon-mixing}.
We refer to \cite{BD22,DPU96,DS06CPAM,Dupont10,PRL07,SUZ14,SUZ15}
for other cases where the CLT
for H{\"o}lder observables 
was
 established in holomorphic dynamics. 
 As is the case for Theorem \ref{t:henon-mixing},
 this is the first time that this is done
 for systems with both positive and negative Lyapunov exponents.

\begin{corollary}\label{c:henon-CLT}
Let $f$ be a complex H\'enon map and $\mu$
 its measure of maximal entropy. Then
all
H\"older observables $u$
satisfy the
Central Limit Theorem  with respect to
 $\mu$
with
\[\sigma^2
=
\sum_{n\in \mathbb Z} \langle\mu, \tilde u (\tilde u \circ f^n) \rangle
=
\lim_{n\to \infty}
\frac{1}{n}\int_{X}
(\tilde u+ \tilde u \circ f + \ldots + \tilde u \circ f^{n-1})^2 d\mu,
\]
where $\tilde u := u-\langle\mu,u\rangle$.
\end{corollary}

\medskip

Theorem \ref{t:henon-mixing} and Corollary \ref{c:henon-CLT}
in particular apply to any real H{\'e}non map
of maximal entropy
\cite{BS04}, 
i.e., complex H{\'e}non maps with real coefficients and 
whose measure of maximal entropy
 is supported by
  $\mathbb R^2$.
They
 hold also in the larger settings of 
regular automorphisms of $\mathbb C^k$ 
in any dimension \cite{Sibony99}, 
see Definition \ref{d:regular-aut} and
Remark \ref{r:autCK},
and 
invertible
 horizontal-like maps in any dimension \cite{DNS08,DS06AIF}, see 
 Remark \ref{r:HL}.
 We postpone
the case of automorphisms of
 compact K\"ahler manifolds to 
the forthcoming paper
 \cite{BD23},
 see Remark \ref{r:Kahler}.

 \medskip

 Our method to prove Theorem \ref{t:henon-mixing}
  relies on pluripotential theory and on the 
 theory of positive closed currents. The idea
 is as follows. Using the classical theory of interpolation \cite{Triebel95}, we can reduce the problem to the case $\gamma=2$. For simplicity, assume that $\|g_j\|_{\Cc^2}\leq 1$ for all $j$.
 The measure of maximal entropy $\mu$
of a H\'enon map $f$ of $\C^2$ of algebraic degree $d\geq 2$
is the intersection $\mu=T_+\wedge T_-$ of the two Green currents  $T_+$ and $T_-$ of $f$
\cite{BS91,Sibony99}.
If we identify $\C^2$ to an affine chart of $\P^2$ in the standard way, these currents are the unique
positive closed 
 $(1,1)$-currents of mass 1 on $\P^2$, without mass at infinity,  
satisfying $f^* T_{+} = d T_{+}$
and
$f_* T_{-} = d T_{-}$.

  \medskip

 Consider the automorphism $F$
 of $\mathbb C^4$ given by $F:=(f,f^{-1})$. 
 Such automorphism also 
 admits Green currents $\mathbb T_+=T_+\otimes T_-$ and $\mathbb T_-=T_-\otimes T_+$. These currents
  satisfy
 $(F^n)^* \mathbb T_+= d^2 \mathbb T_+$ and  $(F^n)_* \mathbb T_-= d^2 \mathbb T_-$.
 Under mild assumptions on their support, 
other positive  closed
$(2,2)$-currents $S$ of mass 1
of $\mathbb P^4$ 
 satisfy the estimate
 \begin{equation}\label{eq:intro-conv-1}
 |\langle d^{-2n} (F^n)_{*} (S) - \mathbb T_-, \Phi\rangle| \leq c_{S,\Phi} d^{-n} 
  \end{equation}
 when 
 $\Phi$
 is a sufficiently smooth  test form. Here, $c_{S,\Phi}$ is a constant depending on $S$ and $\Phi$.

 \medskip
 
We show that proving
 the exponential mixing 
for $\kappa +1$
observables
  $g_0, \dots, g_{\kappa}$ with $\|g_j\|_{\Cc^2}\leq1$ 
can be reduced
to proving
 the convergence (we assume that $n_1$ is even for simplicity)
\begin{equation}\label{eq:intro-goal}
|\langle d^{-n_1} (F^{n_1/2})_{*}[\Delta] -  \mathbb T_-,
\Theta
 \rangle| \lesssim d^{-  \min_{0\leq j \leq \kappa-1} (n_{j+1}-n_j)/2},
\end{equation}
 where $[\Delta]$ denotes the current of integration on the diagonal $\Delta$ of $\C^2\times \C^2$, $(z,w)$ denote
 the
  coordinates on $\C^2\times \C^2$ and 
$$\Theta := g_0(w) g_1(z) (g_2\circ f^{n_2-n_1}(z)) \dots (g_\kappa \circ f^{n_\kappa-n_1}(z)) \mathbb T_+.$$
 A crucial point here is that the estimate should not only be uniform in the $g_j$'s, but also in the $n_j$'s.
Note 
 also
that the current $[\Delta]$ is singular and
the dependence of the constant $c_{S,\Phi}$  in \eqref{eq:intro-conv-1} from $S$
makes it difficult to employ regularization techniques 
 to deduce the
 convergence \eqref{eq:intro-goal}
 from \eqref{eq:intro-conv-1}.

 \medskip
 
 The key point here is to notice that, when $dd^c \Phi \geq 0$ (on a suitable open set),
 one can also
 get the following variation of \eqref{eq:intro-conv-1}:
\begin{equation}\label{eq:intro-conv-2}
 \langle d^{-2n} (F^n)_{*} (S) - \mathbb T_-, \Phi\rangle
  \leq c_{\Phi} d^{-n}.
  \end{equation}
With respect to \eqref{eq:intro-conv-1}, only the bound
from above
is
 present, but the constant $c_\Phi$ is now 
independent of $S$. This permits to regularize $\Delta$ and work as if this current were smooth. Note also that,
although $\Theta$ is not smooth, we can handle it using a similar regularization. 

\medskip

Working by induction, we show that
it is possible to replace both $\Theta$ and
$-\Theta$ in \eqref{eq:intro-goal}
with currents $\Theta^{\pm}$ satisfying $dd^c \Theta^\pm \geq 0$. This permits to deduce
the estimate \eqref{eq:intro-goal} from two upper bounds given by \eqref{eq:intro-conv-2} for
$\Theta^\pm$, completing the proof.

\subsection*{Acknowledgments}
We
would like to thank the National University of Singapore,
the Institut de Math\'ematiques de Jussieu-Paris Rive Gauche,
and Xiaonan Ma
 for the warm welcome and the excellent work conditions.
We also thank
 Romain Dujardin, Livio Flaminio,
 and 
 Giulio Tiozzo  for very useful remarks and discussions.

This project has received funding from
 the French government through the Programme
 Investissement d'Avenir
 (LabEx CEMPI /ANR-11-LABX-0007-01,
ANR QuaSiDy /ANR-21-CE40-0016,
ANR PADAWAN /ANR-21-CE40-0012-01)
and the NUS
and MOE through the grants
A-0004285-00-00
and 
MOE-T2EP20120-0010.

\section{Regular automorphisms of $\mathbb C^k$ and convergence towards Green currents}\label{s:regular}

Let $F$ be a polynomial automorphism of $\mathbb C^k$. We still denote by $F$
its extension  as a birational map of  $\mathbb P^k$. 
Denote by $\H_{\infty}:=\P^k\setminus\C^k$
the hyperplane at infinity and by $\I^+, \I^-$ the indeterminacy sets of $F$ and $F^{-1}$ respectively. They
are analytic sets strictly contained in $\H_\infty$. 
If $\I^+=\varnothing$ or $\I^-=\varnothing$, then both of them
are empty and $F$ is given by a linear map and its dynamics is easy to describe. Hence, we 
assume that $\I^\pm\neq \varnothing$.
The following definition is due to Sibony \cite{Sibony99}.

\begin{definition}\label{d:regular-aut}
We say that $F$ is a \emph{regular automorphism of} $\mathbb C^k$
if $\I^\pm\not=\varnothing$ and $\I^+\cap \I^-= \varnothing$.
\end{definition}

Given $F$ a regular automorphism of $\mathbb C^k$, it is clear that $F^{-1}$ is also regular. We denote
by $d_+(F)$ and $d_-(F)$
the algebraic degrees
 of $F$ and $F^{-1}$ respectively. Observe that $d_{\pm}(F)\geq 2$,  
$d_+ (F)= d_-(F^{-1})$ and $d_- (F)= d_+(F^{-1})$. Later, we will drop the letter $F$ and just write $d_\pm$ instead of $d_\pm(F)$ for simplicity.  
We will recall here some basic properties of $F$ and refer the reader to \cite{BS91,DS14rigidity,Fornaess96,FS92,Sibony99} for details.

\begin{proposition}\label{p:general-aut}
Let $F$ be a regular automorphism of $\mathbb C^k$ as above.
\begin{enumerate}
\item[{\rm (i)}]
There exists an integer $1\leq p\leq k-1$
 such that
$\dim \I^+= k-p-1$, $\dim \I^-= p-1$, and $d_+ (F)^p= d_-(F)^{k-p}$.
\item[{\rm (ii)}] The analytic sets $\I^\pm$ are irreducible and we have 
\[
F(\H_\infty \setminus \I^+) =F(\I^-)= \I^-
\quad \mbox{ and }
\quad 
F^{-1} (\H_\infty \setminus \I^-)= F^{-1}(\I^+)= \I^+.
\]
\item[{\rm (iii)}] For every $n\geq 1$, both $F^{n}$ and $F^{-n}$ are regular automorphisms of $\mathbb C^k$, of algebraic degrees
$d_+(F)^n$ and $d_-(F)^n$,
and indeterminacy sets $\I^+$ and $\I^-$, respectively.
\end{enumerate}
\end{proposition}

\begin{example}\label{ex:Henon}
(Generalized) H\'enon maps
on $\mathbb C^2$  
correspond to the case $k=2$ 
 in Definition  \ref{d:regular-aut}.
 In this case, we have $p=k-p=1$ and $d_+=d_-=d,$ the algebraic degree of 
 the map,
 see \cite{BS91,FM89, Sibony99}.
\end{example}

The set
 $\I^+$ (resp. $\I^-$) is attracting for $F^{-1}$ (resp. $F$).
Let $\widetilde W^{\pm}$ be the basin of attraction of $\I^{\pm}$. Set $W^{\pm}:=\widetilde W^{\pm}\cap \mathbb C^k$.
Then the sets 
 $\K^{+}:= \mathbb C^k \setminus W^{-}$
 and  $\K^{-}:= \mathbb C^k \setminus W^{+}$
 are the sets
 of points (in $\mathbb C^k$)
 with bounded orbit for $F$ and $F^{-1}$, respectively.
We have $\overline {\K^+} = \K^{+}\cup \I^+$ and
$\overline {\K^-} = \K^{-}\cup \I^-$
where the closures are taken in $\mathbb P^k$.
We also define $\K := \K^+ \cap \K^-$ which is a compact subset of  $\C^k$.

\medskip

In the terminology of \cite{DS14rigidity}, the set $\overline{\K^+}$ (resp.\ $\overline{\K^-}$)
is $p$-\emph{rigid} (resp.\ $(k-p)$-\emph{rigid}): it supports a unique positive closed
$(p,p)$-current (resp.\ $(k-p,k-p)$-current) of mass 1, that we denote by $\mathbb T_+$ (resp.\ $\mathbb T_-$). 
The
 currents $\mathbb T_{\pm}$
  have no mass on $\H_\infty$ and satisfy the invariance relations
$$F^*(\mathbb{T}_+)=d_+^p \mathbb{T}_+  \qquad \text{ and  }
\quad
F_* (\mathbb{T}_-)= d_-^{k-p} \mathbb T_-$$
as currents on $\C^k$ or $\P^k$.
We call them {\it the main Green currents of $F$}. They can be obtained as intersections of positive closed $(1,1)$-currents with local H\"older continuous potentials in $\C^k$. Therefore, the measure $\mathbb{T}_+\wedge \mathbb{T}_-$ is well-defined and supported by the compact set $\K$. This is the unique invariant probability 
measure of maximal entropy
\cite{DeThelin10,Sibony99}, see also
\cite{BLS93,BS91,BS92,Dujardin04} for the case of dimension $k=2$.

\medskip

Using the above description of the dynamics of $F$, we can fix neighbourhoods
$U_1,U_2$ of $\overline{\K^+}$ and $V_1,V_2$ of $\overline{\K^-}$
such that $F^{-1} (U_i)\Subset U_i$, $U_1 \Subset U_2\Subset \P^k \setminus \mathbb I^-$, $F(V_i)\Subset V_i$, $V_1 \Subset V_2\Subset \P^k \setminus \mathbb I^+$, 
and $U_2 \cap  V_2\Subset \C^k$.
Let $\Omega$ be a real $(p+1, p+1)$-current with compact support in $U_1$.
Assume that there exists
a positive closed $(p + 1,  p + 1)$-current $\Omega'$ with compact support in $U_1$ 
such that $|\Omega|\leq \Omega'$.
Define the norm $\|\Omega\|_{*,U_1}$
  of $\Omega$ as
\[
\|\Omega\|_{*,U_1} := \inf \{\|\Omega'\|\colon |\Omega|\leq \Omega' \}.
\]
Observe that when $\Omega$ is a
$d$-exact current we can write $\Omega=\Omega'-(\Omega'-\Omega)$,
 which is the difference of two positive closed current 
 in the same cohomology class
 in $H^{p+1,p+1}(\P^k, \mathbb R)$.
 Therefore, the norm $\|\cdot\|_{*,U_1}$ 
 is equivalent to the norm given by $\inf \|\Omega^{\pm}\|$, where
$\Omega^{\pm}$ are positive closed currents with compact support in $U_1$ such that $\Omega = \Omega^+ - \Omega^-$. Note that $\Omega^+$ and $\Omega^-$ have the same mass as they belong to the same cohomology class.

\medskip

The following property was obtained by the second author, see \cite[Proposition 2.1]{Dinh05}.

\begin{proposition}\label{p:acta}
Let $R$ be a positive closed $(k-p,k-p)$-current of mass $1$ with compact support in $V_1$
and smooth on $\mathbb C^k$. 
Let $\Phi$ be a
real-valued $(p,p)$-form of class $\mathcal C^2$
with compact support in $U_1\cap \C^k$. Assume that $dd^c \Phi \geq 0$ on $V_2$.
Then there exists a constant $c>0$
independent of $R$ and $\Phi$ such that
$$\big\langle d_-^{-(k-p)n}  (F^n)_*(R)- \mathbb{T}_-, \Phi \big\rangle
 \leq c \, d_-^{-n}\|\ddc\Phi\|_{*,U_1}
 \quad \mbox{ for all } n\geq 0.$$
\end{proposition}

Note that in what follows,
 since $\mathbb{T}_-$ is an intersection
of positive closed $(1,1)$-currents 
with
local continuous potentials
\cite{Sibony99},
the intersections
$R\wedge \mathbb{T}_-$ and $\mathbb{T}_+\wedge \mathbb{T}_-$ are well-defined
 and the former
 depends continuously on $R$. In particular, the pairing in the next statement
is meaningful and depends continuously on $R$.

\begin{corollary}\label{c:acta}
Let $R$ be a positive closed $(k-p,k-p)$-current of mass $1$ supported in $V_1$.
Let $\phi$ be a $\Cc^2$ function 
with compact support on $\C^k$
such that $\ddc\phi\geq 0$ in a neighbourhood of $\K_+\cap V_2$.
Then there exists a constant $c>0$
independent of $R$ and $\phi$ such that
\begin{equation}\label{eq:c:acta}
\big\langle d_-^{-(k-p)n}  (F^n)_*(R)- \mathbb{T}_-, \phi\mathbb{T}_+ \big\rangle
 \leq c\, d_-^{-n}\|\ddc\phi\wedge \mathbb{T}_+\|_{*,U_1}
\quad \mbox{ for all } n\geq 0.
\end{equation}
\end{corollary}

\proof
As $\mathbb P^k$ is homogeneous, we will use the group $\PGL(k+1,\C)$ of automorphisms of $\P^k$ and suitable convolutions in order regularize the currents $R$ and $\phi\T_+$ and deduce the result from Proposition \ref{p:acta}. Choose local coordinates centered at the identity $\id\in \PGL(k+1,\C)$ so that a small neighbourhood of $\id$ in $\PGL(k+1,\C)$ is identified to the unit ball $\B$ of $\C^{k^2+2k}$. Here, a point of coordinates $\epsilon$ represents an automorphism of $\P^k$ that we denote by $\tau_\epsilon$. Thus, $\tau_0=\id$.

Consider a smooth non-negative function $\rho$ with compact support on $\B$
and of integral 1 with respect to the Lebesgue measure and,
for $0<r\leq 1$,
define
$\rho_r (\epsilon) := r^{-2k^2-4k} \rho (r^{-1}\epsilon)$,
which is supported by $\{|\epsilon|\leq r\}$ . 
This function allows us to define an approximation of the Dirac mass at $0\in\B$ when $r\to 0$.
We define $\Psi:=\phi\T_+$ and consider the following regularized currents
$$R_r:=\int \rho_r(\epsilon) (\tau_\epsilon)^*(R) \quad \text{and} \quad \Psi_r:=\int \rho_r(\epsilon) (\tau_\epsilon)^*(\Psi)
=\int \rho_r(\epsilon) (\phi\circ\tau_\epsilon)(\tau_\epsilon)^*(\T_+),$$
where the integrals are with respect to the Lebesgue measure on $\epsilon\in\B$.

When $r$ is small enough and goes to 0, the current $R_r$ is smooth, positive, closed,
with compact support in $V_1$, and converges to $R$. Since the RHS of \eqref{eq:c:acta} depends continuously on $R$, we can replace $R$ by $R_r$ and assume that $R$ is smooth. When $\epsilon$ goes to 0, $\phi\circ\tau_\epsilon$ converges uniformly to $\phi$ and $(\tau_\epsilon)^*(\T_+)$ converges to $\T_+$. Using that $R$ is smooth and $\T_-$ is a product of $(1,1)$-currents
with
 continuous potentials, we deduce that the LHS of  \eqref{eq:c:acta} is equal to
$$\lim_{r\to 0} \big\langle d_-^{-(k-p)n}  (F^n)_*(R)- \mathbb{T}_-, \Psi_r \big\rangle.$$

Since $\T_+$ is supported by $\overline{\K_+}$ and we have
$\ddc\phi\geq 0$ on a neighbourhood of $\K_+\cap V_2$, we deduce that $\ddc\Psi\geq 0$ on $V_2$. By reducing slightly $V_2$, we still have $\ddc\Psi_r\geq 0$ on $V_2$ for $r$ small enough.
We will use 
the last limit and Proposition \ref{p:acta} for $\Psi_r$ instead of $\Phi$
and $U_2$ instead of $U_1$.
Observe that for $\epsilon$ small enough, since $U_1\Subset U_2$, we have
$\|(\tau_\epsilon)^* (dd^c \Psi)\|_{*,U_2}\leq \|dd^c \Psi\|_{*, U_1}$.
 We deduce
 that the LHS of \eqref{eq:c:acta} is smaller than  or equal to
$$\lim_{r\to 0} c\, d_-^{-n} \|\ddc\Psi_r\|_{*,U_2}
\leq c\, d_-^{-n} \|\ddc\Psi\|_{*,U_1} 
= c\, d_-^{-n} \|\ddc\phi\wedge \T_+\|_{*,U_1}.$$ 
This completes the proof of the corollary.
\endproof

In order to use the above corollary, we will need the following lemmas.

\begin{lemma}\label{l:new-norm-star}
Let  $\kappa \geq 1$ be an integer and
 $g_0, \ldots,g_\kappa$
compactly supported functions on $\C^k$
 with $\|g_j\|_{\Cc^2}\leq 1$. 
 Then there is a constant $c_\kappa>0$ independent of the $g_j$'s
  such that for all $\ell_0, \ldots, \ell_{\kappa}\geq 0$ we have
$$\|dd^c \big((g_0\circ F^{\ell_0})\dots (g_\kappa\circ F^{\ell_\kappa})\big)\wedge \mathbb{T}_+\|_{*,U_1}
\leq c_\kappa.$$
\end{lemma}

\begin{proof}
Set $\tilde g_j := g_j \circ f^{\ell_j}$ for simplicity. We have
\[\begin{aligned}
dd^c \big(\tilde g_0 \dots \tilde g_\kappa\big) = \sum_{j=0}^\kappa dd^c \tilde g_j \prod_{l\neq j} \tilde g_{l}
+
\sum_{0\leq j\neq l \leq \kappa}
 i \partial \tilde g_j \wedge \bar \partial \tilde g_l 
\prod_{m\neq j,l} \tilde g_m.
\end{aligned}
\]
Since
 $\|g_j\|_{\Cc^2}\leq 1$ 
we have
 $|g_j|\leq 1$. Denote by $\omega_{\FS}$ the Fubini-Study form on $\P^k$.
Then
\[
\Big|
\sum_{j=0}^\kappa dd^c \tilde g_j \prod_{l\neq j} \tilde g_{l}
\Big| \ \lesssim \ \sum_{j=0}^\kappa (F^{\ell_j} )^* \omega_{\FS}
\]
and an application of Cauchy-Schwarz inequality gives
\begin{eqnarray*}
\Big| 
\sum_{0\leq j\neq l\leq \kappa}
i \partial \tilde
g_j 
\wedge \overline \partial 
\tilde g_l
\prod_{m\neq j,l} 
\tilde g_m
\Big|
& \lesssim &
  \sum_{j=0}^\kappa
i \partial \tilde g_j \wedge \overline \partial \tilde g_j \\
& = & \sum_{j=0}^\kappa
(F^{\ell_j})^* ( i \partial g_j \wedge \overline \partial  g_j)\\
&\lesssim & \sum_{j=0}^\kappa
(F^{\ell_j})^* (\omega_{\FS}).
\end{eqnarray*}

As we have $dd^c \big(\tilde g_0 \dots \tilde g_\kappa\big) =0$ near $\H_\infty$, 
its intersection with $\mathbb T_+$
can be computed on $\C^k$.
 We deduce from the above inequalities
 and
  $d_-^{k-p}= d_+^{p}$  that
\begin{equation}\label{eq:sum-omega-T}
\begin{aligned}
\big|dd^c \big((g_0\circ F^{\ell_0})\ldots (g_\kappa\circ F^{\ell_\kappa})\big)\wedge \mathbb{T}_+\big|
& \ \lesssim \ 
\sum_{j=0}^\kappa 
\big(F^{\ell_j} )^* (\omega_{\FS} \big)\wedge \mathbb T_+\\
& \ = \ 
\sum_{j=0}^\kappa
(F^{\ell_j} )^* ( \omega_{\FS}) \wedge   d_+^{-p\ell_j} (F^{\ell_j} )^* \mathbb T_+\\
& \ = \ 
\sum_{j=0}^\kappa
d_-^{-(k-p)\ell_j}
 (F^{\ell_j} )^* \big( \omega_{\FS} \wedge \mathbb T_+\big).
\end{aligned}
\end{equation}

We will use that the $(p+1,p+1)$-current $\omega_{\FS}\wedge \mathbb T_+$ is positive,
closed, of mass 1, and its support is 
contained in $\overline { \mathbb {K}^+ } \subset U_1$.
We have
 $$\|(F^{\ell_j} )^* \big( \omega_{\FS} \wedge \mathbb T_+\big)\| = \big \langle (F^{\ell_j} )^* \big( \omega_{\FS} \wedge \mathbb T_+\big), \omega_\FS^{k-p-1} \big \rangle= \big \langle \omega_{\FS} \wedge \mathbb T_+, (F^{-\ell_j})^*(\omega_\FS^{k-p-1}) \big\rangle,$$
where the last form is positive closed and smooth outside $\I^-$.
The last pairing only depends on the cohomology classes of $\omega_{\FS}$, $\mathbb T_+$,
and $(F^{-\ell_j})^* (\omega_{\FS}^{k-p-1})$.
Hence, it is equal to the mass of 
$(F^{-\ell_j})^*(\omega_\FS^{k-p-1})$, which is equal to $d_-^{(k-p-1)\ell_j}$, see \cite{Sibony99}.
It follows that 
each term in 
 the last sum in \eqref{eq:sum-omega-T} is bounded 
by $1$, which implies that the sum is bounded by  
  $\kappa+1$.
The lemma follows.
\end{proof}

\begin{lemma}\label{l:decomposition-g}
Let $D\Subset D'$ be two bounded domains in $\C^k$.
Let $g$ be a function with compact support in $D$
and such that 
$\|g\|_{\Cc^2}\leq 1$. 
Then there
are a constant $A>0$ independent of $g$ and functions $g^\pm$ with compact supports in $D'$ and $\|g^\pm\|_{\Cc^2}\leq 1$ such that
$$g=A(g^+-g^-),  \qquad i\partial g^+\wedge \dbar g^+ \leq \ddc g^+ \text{ on } D \qquad \text{and} \qquad   i\partial g^-\wedge \dbar g^- \leq \ddc g^- \text{ on } D.$$
\end{lemma}

\begin{proof}
Let $\rho$ be a smooth non-negative function, compactly supported on $D'$ and equal to $1$ in a neighbourhood of $\overline D$. Observe that $\rho g=g$.
We denote by $z$ the coordinates of $\C^k$.
Since $\|g\|_{\Cc^2}\leq 1$ and $g$ has compact support in $D$, there exists a constant $A_1 >0$ independent of $g$ 
 such that $|dd^c g| \leq  A_1 \ddc(\|z\|^2)$.
Set $g^+ := A^{-1} \rho ( g + 2A_1 \|z\|^2)$ and $g^- := 2A^{-1}A_1\rho \|z\|^2$ for some constant $A>0$.
It is not difficult to check that
 we have 
 $\ddc g^\pm\geq A^{-1}A_1\ddc (\|z\|^2)$ 
 on $D$,
 $\|i\partial g^\pm\wedge \dbar g^\pm\|_\infty=O(A^{-2})$ on $D$,
 $g = A(g^+-g^-)$,
 and $\|g^{\pm}\|_{\Cc^2}=O(A^{-1})$.
Taking $A$ large enough gives the lemma.
\end{proof}

\section{Exponential mixing of all orders for {H\'e}non maps and further remarks}\label{s:mixing-Henon}

Throughout this section (except for Remarks
\ref{r:non-elementary},
\ref{r:autCK}, and \ref{r:HL}), 
$f$ denotes a H\'enon map on  $\C^2$ of algebraic degree $d=d_+=d_-\geq 2$.
Define $F:=(f, f^{-1})$. It is not difficult to check that $F$ is a regular automorphism of $\C^4=\C^2\times\C^2$.
We will use the notations and the results of Section \ref{s:regular}  
with $k=4$ and $p=2$.
We denote in this section by $T_{\pm}$ the Green $(1,1)$-currents
of $f$, and reserve the notation $\mathbb T_{\pm}$
for the main 
Green currents of $F$. Observe that $\mathbb T_+ = T_+\otimes T_-$ 
and $\mathbb T_- = T_-\otimes T_+$,
see \cite[Section 4.1.8]{Federer} for the tensor (or cartesian)
product of currents.
We denote by $K^{\pm}$ the sets of points of bounded orbit
  for $f^{\pm 1}$.
  The wedge product  $\mu:= T_+\wedge T_-$ is well defined, and is the measure of maximal entropy
  of $f$ \cite{{BLS93,BS91,BS92,Sibony99}}. Its support is contained in the compact set $K=K^+\cap K^-$.
  We have $\K^+=K^+\times K^-$ and $\K^-=K^-\times K^+$. Note also that the diagonal $\Delta$ of $\C^2\times\C^2$ satisfies $\overline\Delta\cap\I^+=\varnothing$ and $\overline\Delta\cap\I^-=\varnothing$ in $\P^4$, see also \cite{Dinh05}.

\medskip

We now
prove Theorem \ref{t:henon-mixing}.
By a standard interpolation \cite{Triebel95}
(see for instance \cite[pp. 262-263]{Dinh05} and 
\cite[Corollary 1]{Dolgopyat98}  for similar occurrences) 
it is enough to prove the statement
for $\gamma=2$, i.e., in the case where all 
the functions $g_j$ are of class $\mathcal C^2$.
The statement is clear for $\kappa=0$, i.e.,
 for one test function.
 By induction, we can assume that the
statement holds for up to $\kappa$ test functions
and prove it for $\kappa+1 \geq 1$ test functions,
i.e., show that 
$$
\Big|
\langle \mu,
 g_0 ( g_1 \circ  f^{n_1} )
\dots
 (g_{\kappa} \circ f^{n_{\kappa}})
\rangle
- \prod_{j=0}^{\kappa} \langle  \mu, g_j
\rangle
\Big|
\lesssim \Big(
\prod_{j=0}^{\kappa} 
\|g_j\|_{\Cc^2}
\Big) \cdot
d^{-\min_{0\leq j \leq \kappa-1} ( n_{j+1}- n_j)/2}.
$$
Recall that $n_0=0$. The induction assumption implies that we are allowed to modify each $g_j$ by adding a constant. 
Moreover, using the invariance of $\nu$, the desired estimate does
not change if we replace $n_j$ by $n_j-1$
 for $1\leq j\leq \kappa$
  and $g_0$ by $g_0 \circ f^{-1}$. 
Therefore, we can for convenience assume that $n_1$ is even.

\medskip

We fix 
a large bounded domain 
$B\subset \mathbb C^2$ satisfying 
\begin{equation*}
K \subset B, \quad
K^-\cap B \subset f(B),
\quad
\mbox{ and }
\quad
K^+ \cap B \subset f^{-1} (B). 
\end{equation*}
By induction, the inclusions above imply that
\begin{equation}\label{eq:properties-neigh-K}
K \subset B, \quad
K^-\cap B \subset f^n(B),
\quad
\mbox{ and }
\quad
K^+ \cap B \subset f^{-n} (B)
\quad \mbox{ for all } n \geq 1. 
\end{equation}
Because of Lemma \ref{l:decomposition-g}
and the fact that we 
are only interested in the values of the $g_j$'s on the support of $\mu$,
we
can assume that all the $g_j$'s
are compactly supported in $\mathbb C^2$
and  satisfy
\begin{equation}\label{eq-simpl-g}
\|g_j\|_{\Cc^2}\leq 1
\mbox{ on } \mathbb C^2\quad \mbox{ and } 
\quad i\partial g_j \wedge \overline \partial g_j\leq dd^c g_j \mbox{ on } B.
\end{equation}
For simplicity, write
$h:= g_1 (g_2 \circ f^{n_2-n_1} ) \dots ( g_{\kappa} \circ f^{n_{\kappa}-n_1})$.
We need to prove that
\begin{equation*} 
|\langle \mu, g_0  ( h\circ f^{n_1}) 
\rangle - \langle\mu, g_0\rangle \cdot \langle\mu, h \rangle|\lesssim  d^{- \min_{0\leq j \leq \kappa-1} (n_{j+1}-n_j)/2}
\end{equation*}
since this estimate, together with the induction assumption applied to $\langle\mu,h\rangle$, 
would imply the desired statement. In order to obtain the result, we will prove separately the 
two
 estimates
\begin{equation}\label{eq:goal-mixing-1}
\langle \mu, g_0  ( h\circ f^{n_1}) 
\rangle - \langle\mu, g_0\rangle \cdot \langle\mu, h \rangle \lesssim  d^{- \min_{0\leq j \leq \kappa-1} (n_{j+1}-n_j)/2}
\end{equation}
and
\begin{equation}\label{eq:goal-mixing-2}
-\langle \mu, g_0  ( h\circ f^{n_1}) 
\rangle + \langle\mu, g_0\rangle \cdot \langle\mu, h \rangle \lesssim  d^{- \min_{0\leq j \leq \kappa-1} (n_{j+1}-n_j)/2}.
\end{equation}
 
 \medskip
 
Set $M:=10 \kappa$
and fix a smooth function 
$\chi$ with compact support in $\C^2$ and equal to $1$ in a neighbourhood of $\overline{B}$.
Consider the following four functions,
which will later allow us to produce some p.s.h. test functions:
$$g_0^+:=\chi\cdot (g_0+M ) \quad \text{and} \quad 
h^+:=\chi\cdot  (g_1+M)
(g_2 \circ f^{n_2-n_1} + M)
\dots
(g_\kappa\circ f^{n_\kappa-n_1}+M)$$ 
and 
$$g_0^-:=\chi\cdot(M-g_0) 
\quad \text{and} \quad 
h^-:=\chi\cdot 
\big((g_1+M)
(g_2 \circ f^{n_2-n_1} + M)\dots(g_\kappa\circ f^{n_\kappa-n_1}+M)-2(M+1)^\kappa\big).$$  
Recall that $n_0=0$. To prove \eqref{eq:goal-mixing-1} and \eqref{eq:goal-mixing-2}, it is enough to show that 
\begin{equation}\label{eq:goal-mixing-3}
\langle \mu, g_0^+  ( h^+\circ f^{n_1}) 
\rangle - \langle\mu, g_0^+\rangle \cdot \langle\mu, h^+ \rangle \lesssim  d^{- n_1/2}
\end{equation}
and
\begin{equation}\label{eq:goal-mixing-4}
\langle \mu, g_0^-  ( h^-\circ f^{n_1}) 
\rangle - \langle\mu, g_0^-\rangle \cdot \langle\mu, h^- \rangle \lesssim  d^{-n_1/2}.
\end{equation}
Indeed, we observe that
$\chi$ does not play any role in 
\eqref{eq:goal-mixing-3} and \eqref{eq:goal-mixing-4}. Hence,
the difference between the LHS of \eqref{eq:goal-mixing-3} and the one of \eqref{eq:goal-mixing-1} (resp. of \eqref{eq:goal-mixing-4} and of \eqref{eq:goal-mixing-2}) is a finite combination of expressions involving no more than
$\kappa$
 functions among $g_0,\ldots, g_\kappa$,
 that we can estimate using the induction hypothesis on the mixing of order up to $\kappa-1$. 
 It remains to prove the two inequalities
 \eqref{eq:goal-mixing-3} and \eqref{eq:goal-mixing-4}.

\medskip

Denote by $(z,w)$
the coordinates on $\C^{4} = \mathbb C^2 \times \mathbb C^2$ and define
$$\phi^\pm(z,w):= \, g_0^\pm(w)\, h^\pm(z).$$
We have the following lemma for a fixed domain $U_1$ as in Section \ref{s:regular}.

\begin{lemma}\label{l:phi}
The functions $\phi^\pm$
satisfy
\begin{enumerate}
\item[{\rm (i)}] $dd^c \phi^\pm  \wedge \mathbb T_+\geq 0 $ on $B\times B$;
\item[{\rm (ii)}] $\|dd^c \phi^\pm \wedge \mathbb T_+\|_{*,U_1} \leq c_\kappa$,
\end{enumerate}
where $c_\kappa$ is a positive constant depending on $\kappa$, but not on the $g_j$'s and the $n_j$'s.
\end{lemma}

\begin{proof}
(i) For simplicity, we set $\ell_0 = \ell_1 :=0$ and $\ell_j := n_j-n_1$. Define also
$\tilde g_{j} := g_j \circ f^{\ell_j}$. In what follows $\tilde g_0$ depends on $w$ and $\tilde g_j$ depends on $z$ when $j\geq 1$.
Observe that by the invariance property of $K^+\cap B$ in \eqref{eq:properties-neigh-K} and the constraints in \eqref{eq-simpl-g}, the following inequalities
 hold in a neighbourhood of $K^+\cap B$:
\begin{equation}\label{e:ddb-ddc}
i\partial \tilde g_j\wedge \dbar \tilde g_j = (f^{\ell_j})^*(i\partial g_j\wedge \dbar g_j) \leq (f^{\ell_j})^*(\ddc g_j) = \ddc \tilde g_j.
\end{equation}
In particular, we have $\ddc \tilde g_j\geq 0$ in a neighbourhood of $K^+\cap B$.
Note that for $\tilde g_0 = g_0$ the properties hold on $B$, which contains $K^-\cap B$.

Now, since $\mathbb T_+$ is closed, positive, and supported by $\overline{\mathbb K^+} =\overline{K^+\times K^-}$, in order to prove the first assertion
 it is enough to show
that $dd^c \phi^\pm\geq 0$ on a neighbourhood of $(K^+\cap B)\times (K^-\cap B)$ in $\C^4$
where $\chi=1$. 
In what follows, we only work on such a neighbourhood. 
We have
\[
dd^c \phi^+ 
=
\sum_{j=0}^{\kappa} dd^c \tilde g_j
\prod_{l\neq j} (\tilde g_l
+ M)
+
\sum_{0\leq  j\neq l \leq \kappa}
i \partial \tilde 
g_j
\wedge \overline \partial 
\tilde g_l
\prod_{m\neq j,l} (
\tilde g_m
+ M),
\]
where we recall that $\tilde g_0$ is $\tilde g_0 (w)$
and the other $\tilde g_j$'s are $\tilde g_j (z)$ for $1\leq j \leq \kappa$.
For the first term in the 
RHS of the last expression, we have 
$$\sum_{j=0}^{\kappa} dd^c \tilde g_j \prod_{l\neq j} (\tilde g_l + M) \geq (M-1)^{\kappa} \sum_{j=0}^{\kappa} dd^c \tilde g_j .$$
For the second term, an application of Cauchy-Schwarz inequality
and \eqref{e:ddb-ddc} give
\[
\begin{aligned}
\big| 
\sum_{0\leq j\neq l \leq \kappa}
i \partial \tilde
g_j
\wedge \overline \partial 
\tilde g_l
\prod_{m\neq j,l} (
\tilde g_m
+ M)
\big|
& \leq
(\kappa+1) (M+1)^{\kappa-1} \sum_{j=0}^{\kappa}
i \partial \tilde g_j \wedge \overline \partial \tilde g_j \\
&\leq
(\kappa+1) (M+1)^{\kappa-1} \sum_{j=0}^\kappa
dd^c \tilde g_j.\\
\end{aligned}\]
It follows that
\begin{eqnarray*}
dd^c \phi^+ &\geq& 
(M-1)^{\kappa} \sum_{j=0}^{\kappa} dd^c \tilde g_j
- (\kappa+1) (M+1)^{\kappa-1} \sum_{j=0}^\kappa dd^c \tilde g_j \\
&=&(M-1)^{\kappa}\Big[1-{(\kappa+1)\over M+1}\big(1+{2\over M-1}\big)^\kappa\Big] \sum_{j=0}^{\kappa} dd^c \tilde g_j,
\end{eqnarray*}
which gives $dd^c \phi^+ \geq 0$ since the choice
 $M=10\kappa$ implies that
  $(1+{2\over M-1}\big)^\kappa<(1+{1\over \kappa}\big)^\kappa<3$.
Similarly, in the same way,
we also have
\[
dd^c \phi^-
\geq 
(M-1)^\kappa \sum_{j=0}^\kappa
dd^c \tilde g_j 
- (M+1)^{\kappa-1} \sum_{0\leq j\neq  l\leq \kappa} \big|
i \partial \tilde
g_j
\wedge \overline \partial 
\tilde g_l
\big|,
\]
which gives $dd^c \phi^-\geq 0$.
This concludes the proof of the first assertion of the lemma. 

\medskip

(ii) The second assertion of the lemma is a consequence of Lemma \ref{l:new-norm-star}.
\end{proof}

\proof[End of the proof of Theorem \ref{t:henon-mixing}]
Recall that it remains to prove  \eqref{eq:goal-mixing-3} and \eqref{eq:goal-mixing-4}. 
Since $\mu$ is invariant and $n_1$ is even, we have
\[
\langle \mu, g_0^\pm\cdot ( h^\pm\circ f^{n_1})  
\rangle
=
\big\langle  \mu,
( g_0^\pm \circ f^{-n_1/2} )( h^\pm\circ f^{n_1 /2}) \big\rangle.
\]
We now transform the last integral on $\mathbb C^2$ to an integral on $\mathbb C^{4}$ in order to use the dynamical system $F=(f,f^{-1})$ on $\C^{4}$
introduced above.
We are using the coordinates $(z,w)$ on $\C^{4} = \mathbb C^2 \times \mathbb C^2$. We will also use the diagonal of $\C^2\times \C^2$,
 which is given by $\Delta=\{(z,w)\colon z=w\}$. 

Recall that we have $\mu=T_+\wedge T_-$ and that
 the currents $T_\pm$ have local continuous potentials in $\C^2$. It follows that 
the intersections of $\T_\pm$ with positive closed currents on $\C^4$ are meaningful.
Moreover, the invariance of $T_\pm$ implies that $(F^{n_1/2})_*(\T_+)=d^{-n_1}\T_+$ on $\C^4$.
Thanks to the above identities, we have
\begin{equation}\label{eq:develop-mixing}
\begin{aligned}
\langle \mu, g_0^\pm \cdot ( h^\pm\circ f^{n_1})   \rangle
& =
\big\langle 
T_+ \wedge T_-,  (g_0^\pm \circ f^{-n_1/2} )(h^\pm \circ f^{n_1/2}) \big\rangle\\
& =
\big\langle
(T_+ \otimes  T_-) \wedge [\Delta],
(g_0^\pm \circ f^{-n_1/2} (w) )(
 h^\pm\circ f^{n_1 /2} (z)) \big\rangle\\
 &=
\big \langle
\mathbb T_+  \wedge [\Delta],
(F^{n_1/2})^*(\phi^\pm) \big\rangle
 \\
 &=
 \big\langle
d^{-n_1} \mathbb T_+\wedge  (F^{n_1/2})_* [\Delta],
 \phi^\pm  \big\rangle\\
 & =
 \big \langle
d^{-n_1} (F^{n_1/2})_* [\Delta],
 \phi^\pm   \mathbb T_+ \big\rangle.
 \end{aligned}
\end{equation}

We apply 
Corollary \ref{c:acta}
with the functions 
$\phi^\pm=g_0^\pm(w) \cdot  h^\pm (z)$ 
instead of $\phi$
and the current $[\Delta]$ instead of $R$. For this purpose, since $\overline\Delta\cap\I^+=\varnothing$, we can choose a suitable open set $V_1$ containing $\Delta$. We 
also fix an open set
$V_2$ as in Section \ref{s:regular}.
Since $B$ is large enough,
Lemma \ref{l:phi} implies that $\ddc\phi^\pm\geq 0$ on a neighbourhood of $\K^+\cap V_2$. Thus, we obtain from Corollary \ref{c:acta} that
\[
 \langle
d^{-n_1} (F^{n_1/2})_* [\Delta] - \mathbb T_-,
 \phi^{\pm}   \mathbb T_+ \rangle \lesssim  d^{- n_1/2}
\]
or equivalently
\begin{equation}\label{eq:from-cor-equiv}
 \langle d^{-n_1} (F^{n_1/2})_* [\Delta],
 \phi^{\pm}
  \mathbb T_+ \rangle 
 - \langle \T_-, 
 \phi^{\pm} 
   \mathbb T_+ \rangle\lesssim  d^{- n_1/2}.
\end{equation}
Together,
\eqref{eq:develop-mixing}, \eqref{eq:from-cor-equiv},
and the fact that 
\[\langle \mathbb T_-,  
\phi^{\pm}
 \mathbb T_+\rangle =\langle \T_+\wedge \T_-, 
 \phi^{\pm}\rangle 
 =\langle\mu \otimes \mu,  g_0^\pm(w) \cdot 
 h^\pm (z)\rangle = \langle\mu, g_0^\pm\rangle \cdot \langle \mu, h^\pm\rangle \]
 give the desired estimates \eqref{eq:goal-mixing-3} and \eqref{eq:goal-mixing-4}. 
The proof of Theorem \ref{t:henon-mixing} is complete.
\endproof

\begin{remark}\label{r:non-elementary} \rm
Friedland and Milnor \cite{FM89}
proved that any polynomial automorphism of $\C^2$ is either conjugate
to an elementary automorphism which preserves a
 fibration
by parallel complex lines or to a H\'enon map as above. So, our results
apply to all automorphisms of $\C^2$ which are not conjugated to an  elementary one.
\end{remark}

\begin{remark}\label{r:autCK} \rm
When $f$ is a regular automorphism of $\C^k$ with $k$ even
and $p=k/2$, the map $(f,f^{-1})$ is regular on $\C^{2k}$.
The same proof as above gives us the exponential mixing of all orders
and the CLT for $f$.
 The results still hold for every regular automorphism but the proof
 requires some extra technical arguments
  that we choose to do not  present here for simplicity, see for instance
  de Thélin and Vigny
   \cite{DTV10, Vigny15}.
\end{remark}

\begin{remark}\label{r:HL} \rm
When $f$ is a horizontal-like map 
such that the main dynamical degree is larger than the other dynamical degrees,
 the same strategy gives the exponential mixing of all orders
 and the CLT, 
 see the papers
\cite{DNS08,DS06AIF}
by Nguyen, Sibony and the second author, and in particular
 \cite{DNS08}
  for the necessary estimates.
   In particular, these results hold for all H{\'e}non-like maps in dimension 2, 
   see also Dujardin
    \cite{Dujardin04}.
\end{remark}

\begin{remark}\label{r:Kahler} \rm
 In the companion paper \cite{BD23},
 we explain how to adapt our strategy to get the
 exponential mixing of all orders and the Central Limit Theorem for 
automorphisms of compact K\"ahler manifolds with simple action on cohomology.
As the proof in that case requires the theory of super-potentials, which is not needed for H{\'e}non maps,
we choose to do not 
present it here.
\end{remark}

\end{document}